\documentclass[11pt]{article}
\usepackage{latexsym,amsfonts,amssymb,amsmath,amsthm}
\usepackage{graphicx}

\usepackage[usenames,dvipsnames]{color}
\usepackage{ulem}

\usepackage{color}

\begin{document}

\newtheorem{tm}{Theorem}[section]
\newtheorem{prop}[tm]{Proposition}
\newtheorem{defin}[tm]{Definition}
\newtheorem{coro}[tm]{Corollary}
\newtheorem{lem}[tm]{Lemma}
\newtheorem{assumption}[tm]{Assumption}
\newtheorem{rk}[tm]{Remark}
\newtheorem{nota}[tm]{Notation}
\numberwithin{equation}{section}

\newcommand{\stk}[2]{\stackrel{#1}{#2}}
\newcommand{\dwn}[1]{{\scriptstyle #1}\downarrow}
\newcommand{\upa}[1]{{\scriptstyle #1}\uparrow}
\newcommand{\nea}[1]{{\scriptstyle #1}\nearrow}
\newcommand{\sea}[1]{\searrow {\scriptstyle #1}}
\newcommand{\csti}[3]{(#1+1) (#2)^{1/ (#1+1)} (#1)^{- #1
 / (#1+1)} (#3)^{ #1 / (#1 +1)}}
\newcommand{\RR}[1]{\mathbb{#1}}

\newcommand{\rd}{{\mathbb R^d}}
\newcommand{\ep}{\varepsilon}
\newcommand{\rr}{{\mathbb R}}
\newcommand{\alert}[1]{\fbox{#1}}
\newcommand{\eqd}{\sim}
\def\p{\partial}
\def\R{{\mathbb R}}
\def\N{{\mathbb N}}
\def\Q{{\mathbb Q}}
\def\C{{\mathbb C}}
\def\l{{\langle}}
\def\r{\rangle}
\def\t{\tau}
\def\k{\kappa}
\def\a{\alpha}
\def\la{\lambda}
\def\De{\Delta}
\def\de{\delta}
\def\ga{\gamma}
\def\Ga{\Gamma}
\def\ep{\varepsilon}
\def\eps{\varepsilon}
\def\si{\sigma}
\def\Re {{\rm Re}\,}
\def\Im {{\rm Im}\,}
\def\E{{\mathbb E}}
\def\P{{\mathbb P}}
\def\Z{{\mathbb Z}}
\def\D{{\mathbb D}}
\newcommand{\ceil}[1]{\lceil{#1}\rceil}

\title{Spreading speeds of a parabolic-parabolic chemotaxis model with  logistic source on $\mathbb{R}^{N}$}

\author{
Wenxian Shen and Shuwen Xue \\
Department of Mathematics and Statistics\\
Auburn University, Auburn, AL 36849,
U.S.A.\\ 
$\quad$\\
Dedicated to Professor Georg Hetzer on the occasion of his  75th Birthday }

\date{}
\maketitle

\begin{abstract}
The current paper is concerned with the spreading speeds of the following parabolic-parabolic chemotaxis model with logistic source on $\R^{N}$,
\begin{equation}\label{Abstract-Eq}
\begin{cases}
u_{t}=\Delta u - \chi\nabla\cdot(u\nabla v)+ u(a-bu),\quad  x\in\R^N, \\
{v_t}=\Delta v-\lambda v+\mu u,\quad x\in\R^N,
\end{cases}
\end{equation}
where $\chi, \ a,\  b,\ \lambda,\ \mu$ are positive constants. Assume $b>\frac{N\mu\chi}{4}$.  Among others, it is proved that $2\sqrt{a}$ is the spreading speed of the global classical solutions of \eqref{Abstract-Eq} with nonempty compactly supported initial functions, that is,
$$
\lim_{t\to\infty}\sup_{|x|\geq ct}u(x,t;u_0,v_0)=0\quad \forall\,\, c>2\sqrt{a}
$$
and
$$
\liminf_{t\to\infty}\inf_{|x|\leq ct}u(x,t;u_0,v_0)>0 \quad \forall\,\, 0<c<2\sqrt{a}.
$$
where $(u(x,t;u_0,v_0), v(x,t;u_0,v_0))$ is the unique global classical solution of \eqref{Abstract-Eq} with $u(x,0;u_0,v_0)=u_0$,
$v(x,0;u_0,v_0)=v_0$, and ${\rm supp}(u_0)$, ${\rm supp}(v_0)$ are nonempty and compact. It is well known that $2\sqrt{a}$ is the spreading speed of the following Fisher-KPP equation,
$$
u_t=\Delta u+u(a-bu),\quad \forall\,\ x\in\R^N.
$$
Hence, if $b>\frac{N\mu\chi}{4}$, the chemotaxis neither speeds up nor slows down the spatial spreading in the Fisher-KPP equation.

\end{abstract}



\section{Introduction and the Statements of Main results}

Chemotaxis is referred to the directed movement of cells and organisms in response to chemical gradients  and plays a crucial role in a wide range of biological phenomena \cite{MEis}. Positive chemotaxis occurs if the movement is toward a higher concentration of the chemical substance in question. Conversely, negative chemotaxis occurs if the movement is in the opposite direction. The chemical substances that lead to positive chemotaxis are called chemoattractants and those leading to negative chemotaxis are called chemorepellents.
Mathematical models for chemotaxis date to the pioneering works of Keller and Segel in the 1970s  \cite{KeSe1, KeSe2}.The reader is referred to \cite{BBTW, HiPa, Hor, KJPainter} and the references therein for some detailed introduction into the mathematics and applications of chemotaxis models. For the recent developments on chemotaxis models, we refer to the survey paper \cite{ArTy}.

The current paper is devoted to the study of the spatial spreading dynamics of the following parabolic-parabolic chemotaxis model with logistic source on
 $\R^N$:
\begin{equation}\label{Main-Eq}
\begin{cases}
u_{t}=\Delta u - \chi\nabla \cdot (u\nabla v)+ u(a-bu),\quad  x\in\R^N,\\
{v_t}=\Delta v -\lambda v+\mu u,\quad x\in\R^N,
\end{cases}
\end{equation}
where $\chi$, $a$, $b$, $\lambda$ and $\mu$ are positive constants. In \eqref{Main-Eq}, $u(x, t)$and $v(x, t)$ denote the population densities of some biological species and chemical substance at location $x$ and time $t$, respectively; the term $\Delta u$ describes the movement of the biological species  following random walk; the term $ \chi\nabla \cdot (u\nabla v)$ characterizes the influence of chemical substance, and the logistic term $u(a-bu)$ governs the local dynamics of the biological species. The second equation indicates that the chemical substance diffuses via random walk with a finite diffusion rate and is produced over time by the biological species.
Both mathematically and biologically, it is important to investigate how chemotaxis affects the dynamics of \eqref{Main-Eq}.

Numerous  research works  have been carried out on the dynamics of the following  counterpart of \eqref{Main-Eq} on a bounded domain with Neumann boundary condition,
\begin{equation}\label{Main-Eq-1}
\begin{cases}
u_{t}=\Delta u - \chi\nabla \cdot (u\nabla v)+ u(a-bu),\quad  x\in \Omega,\\
v_t=\Delta v -\lambda v+\mu u,\quad x\in\Omega,\\
\frac{\p u}{\p n}=\frac{\p v}{\p n}=0,\quad x\in\p \Omega,
\end{cases}
\end{equation}
where $\Omega\subset \R^N$ is a bounded smooth domain (see \cite{DiHo, HeVe1, IsSh2, kuto_PHYSD, lankeit_eventual, LiMu, win_CPDE2010, win_JMPA, win_JDE2014, ZhLiBaZo}, etc.).  For example, when $a\equiv b \equiv0$ in \eqref{Main-Eq-1},  and $\Omega$ is a ball in $\R^N$ with $N\geq 3$, it is proved that for any $M>0$ there exists positive initial data $(u_0,v_0)\in C(\bar\Omega)\times W^{1,\infty}(\Omega)$ with $\int_{\Omega}u_0=M$ such that the corresponding solution blows up in finite time
(see \cite{win_JMPA}). It is shown in \cite{win_JDE2014} that, when $\Omega$ is a convex bounded domain with smooth boundary and $\frac{b}{\chi}$ is sufficiently large, for any choice of suitably regular nonnegative initial data $(u_0, v_0)$ such that $u_0\not\equiv 0$, \eqref{Main-Eq-1} possesses a uniquely determined global classical solution and that the constant solution $(\frac{a}{b},\frac{\mu a}{\lambda b})$ is asymptotically stable in the sense that
 \[
 \lim_{t\to\infty} \big[ \|u(\cdot,t;u_0,v_0)-\frac{a}{b}\|_{L^\infty(\Omega)}+\|v(\cdot,t;u_0,v_0)-\frac{\mu a}{\lambda b}\|_{L^\infty(\Omega)}\big]=0.
 \]
The particular requirement on the convexity of the bounded domain $\Omega$ was later removed in \cite{IsSh2} and \cite{ZhLiBaZo}. Hence finite-time blow-up phenomena in \eqref{Main-Eq-1} can be suppressed to some extent by the logistic source.

There are also numerous  research papers on the dynamics of \eqref{Main-Eq} in the case that $a\equiv b \equiv0$ and many interesting dynamical scenarios are observed
(see \cite{CsPj, CoEsma, FePr, Mi, NtSrUm, NtYt, NaSuYo} etc.). For example, it is observed that finite-time blow-up may occur when $N\ge 2$
(see  \cite{CsPj}) and  it is shown that bounded solutions decay to zero as time goes to infinity and behaves like the heat kernel with the self-similarity (see \cite{NtSrUm, NtYt}).

 Very recently, the authors of current paper \cite{ShXu} studied the dynamics of \eqref{Main-Eq} with $a,b>0$  and observed that
finite-time blow-up phenomena in \eqref{Main-Eq} can also be suppressed to some extent by the logistic source.
To be more precise,  it is proved in   \cite{ShXu}  that if $b>\frac{N\mu\chi}{4}$, \eqref{Main-Eq} has a unique bounded global classical solution for every nonnegative, bounded, and uniformly continuous function $u_0(x)$, and every nonnegative, bounded, uniformly continuous, and differentiable function $v_0(x)$. Moreover, any globally defined bounded positive classical solution with strictly positive initial function $u_0$ is bounded below by a positive constant independent of $(u_0, v_0)$ when time is large.

The objective of the current paper is to investigate the spreading speeds of globally defined  classical solutions of \eqref{Main-Eq} with compactly supported or front-like initial functions.   Roughly speaking,
it is about how fast the biological species spreads into the region without biological species initially as time evolves.

Observe that, in the absence of chemotaxis (i.e. $\chi= 0$), \eqref{Main-Eq} reduces to the following reaction-diffusion equation
\begin{equation}\label{fisher-kpp}
u_t=\Delta u+u(a-bu), \quad x\in\R^N.
\end{equation}
Due to the pioneering works of Fisher \cite{Fisher} and Kolmogorov, Petrowsky, Piskunov \cite{KPP} on traveling wave solutions and take-over properties of \eqref{fisher-kpp}, \eqref{fisher-kpp} is also referred to as  the Fisher-KPP equation.
The following results are well known about the spatial spreading dynamics of \eqref{fisher-kpp}.
Equation \eqref{fisher-kpp} has  traveling wave solutions $u(t,x)=\phi(x\cdot\xi-ct)$ ($\xi\in S^{N-1}$)
connecting $\frac{a}{b}$ and $0$ $(\phi(-\infty)=\frac{a}{b},\phi(\infty)=0)$ of all speeds $c\geq 2\sqrt a$ and has no such traveling wave solutions of slower speeds.  For any given bounded $u_0\in C(\R^N, \R^{+})$ with $\liminf_{x\cdot\xi\to-\infty}u_0(x)>0$ and $u_0(x)=0$ for $x\cdot\xi\gg 1$,
$$
\lim_{t\to\infty}\sup_{x\cdot\xi \ge ct}u(x, t)=0 \quad \forall \, c>2\sqrt a
$$
and
$$\lim_{t\to \infty}\sup_{x\cdot\xi \le ct}|u(x, t)-\frac{a}{b}|=0\quad \forall\,  c<2\sqrt a.
$$
Since their pioneering works,  a considerable amount of research has been carried out toward the front propagation dynamics of
 reaction-diffusion equations of the form,
\begin{equation}
\label{general-fisher-eq}
u_t=\Delta u+u f(x, t, u),\quad x\in\R^N,
\end{equation}
where $f(x, t, u)<0$ for $u\gg 1$,  $\partial_u f(x, t, u)<0$ for $u\ge 0$ (see \cite{ArWe2, Berestycki1, LiZh1, She1, Wei2}, etc.).
In literature, the number $c_0^*:=2\sqrt a$ is called the {\it
spreading speed} for \eqref{fisher-kpp} which was first introduced by Aronson and Weinberger \cite{ArWe1}.

It is interesting  to investigate the influence of chemotaxis on the spreading dynamics of \eqref{Main-Eq}.   The authors of  \cite{SaSh3}, \cite{SaSh4}
studied  the existence of traveling wave solutions of  \eqref{Main-Eq}.
Among others, it is proved that if $b>2\chi\mu$ and $1 \geq \frac{1}{2}(1-\frac{\lambda}{a})_{+} ,$ then for every $c\ge 2\sqrt{a}$, \eqref{Main-Eq} has a traveling wave solution $(u,v)(t,x)=(U^{c}(x\cdot\xi-ct),V^{c}(x\cdot\xi-ct))$ ($\forall\, \xi\in S^{N-1}$)
connecting the two constant steady states $(0,0)$ and $(\frac{a}{b},\frac{\mu}{\lambda}\frac{a}{b})$,
 and there is no such solution with speed $c$ less than $2\sqrt{a}$, which shows that \eqref{Main-Eq} has a  minimal wave speed $c_0^*=2\sqrt a$, which is independent of the chemotaxis.

In this paper, we will prove that,  if $b>\frac{N\mu\chi}{4}$, then $c_0^*=2\sqrt a$ is also the spatial spreading speed of \eqref{Main-Eq},
which shows that  the chemotaxis neither speeds up nor slows down the spatial spreading in the Fisher-KPP equation \eqref{fisher-kpp}. To state our results precisely, we introduce some notations.
Let
\begin{equation*}
\label{unif-cont-space}
X_1=C_{\rm unif}^b(\R^N):=\{u\in C(\R^N)\,|\, u(x)\,\,\text{is uniformly continuous in}\,\, x\in\R^N\,\, {\rm and}\,\, \sup_{x\in\R^N}|u(x)|<\infty\}
\end{equation*}
equipped with the norm $\|u\|_\infty=\sup_{x\in\R^N}|u(x)|$, and
$$
X_2=C_{\rm unif}^{b,1}:=\{u\in C_{\rm unif}^b(\R^N)\,|\, \p_{x_i}u\in C_{\rm unif}^b(\R^N),\,\, i=1,2,\cdots, N\}
$$
equipped with the norm $\|u\|_{C_{\rm unif}^{b, 1}}=\|u\|_\infty+\sum_{i=1}^N \|\p_{x_i}u\|_\infty$.
Let
$$
X_1^+=\{u\in X_1\,|\, u\ge 0\},\quad X_2^+=\{v\in X_2\,|\, v\ge 0\}.
$$

For any given $(u_0, v_0)\in X_{1}^{+}\times X_{2}^{+}$, we denote by $(u(x, t ;u_0, v_0),v(x, t ; u_0, v_0))$ the classical solution of \eqref{Main-Eq} satisfying $u(x, 0;u_0, v_0)=u_0(x)$ and $v(x, 0;u_0, v_0)=v_0(x)$ for every $x\in\R^N$. Note that, by the comparison principle for parabolic equations, for every $(u_0, v_0)\in X_{1}^{+}\times X_{2}^{+}$, it always holds that $u(x, t ;u_0,v_0)\geq 0$ and $v(x, t;u_0, v_0)\geq 0$ whenever $(u(x, t;u_0, v_0),v(x, t;u_0, v_0))$ is defined. In this work we shall only focus on nonnegative classical solutions of \eqref{Main-Eq}  since both functions $u(x, t)$ and $v(x, t)$ represent density functions.

The following proposition states the existence and uniqueness of global classical solutions of \eqref{Main-Eq} with non-negative initial function. It has been proved in \cite[Theorem 1.2]{ShXu}.

\begin{prop}
Suppose that $b>\frac{N\mu\chi}{4}$. Then for every $(u_0, v_0)\in X_{1}^{+}\times X_{2}^{+}$,  \eqref{Main-Eq} has a unique bounded global classical solution $(u(x,t;u_0, v_0)$, $v(x,t;u_0, v_0))$.
\end{prop}

For given $x=(x_1, x_2, \cdot \cdot \cdot, x_N)\in \R^N$, let $|x|=\sqrt{x_1^2+x_2^2+\cdot \cdot \cdot+x_{N}^2}$.
Let
$$
S^{N-1}=\{x \in\R^N\,|\, |x|=1\}.
$$
For $x=(x_1, x_2, \cdot \cdot \cdot, x_N)\in\R^N$, $y=(y_1, y_2, \cdot \cdot \cdot, y_N)\in \R^N$, define $x\cdot y=\sum_{i=1}^{N}x_{i}y_{i}$.

Let
$$
C_{cp}^+=\{u\in X_1^+\,|\,\,\,  {\rm supp}(u)\,\,\, \text{is non-empty and compact}\},
$$
and
$$
C_{cp}^{+,1}=\{v\in X_2^+\,|\, \,\, {\rm supp}(v)\,\,\, \text{is non-empty and compact}\}.
$$
For any given $\xi\in S^{N-1}$, we define
$$
C_{fl}^+(\xi)=\{u\in X_1^+\,\,|\, \, \liminf_{x\cdot\xi\to -\infty}u(x)>0,\,\, u(x)=0\,\, {\rm for}\,\, x\in\R^N\,\ {\rm with}\,\ x\cdot\xi\gg 1\},
$$
$$
C_{fl}^{+,1}(\xi)=\{v\in X_2^+\,\,|\,\, \liminf_{x\cdot\xi\to -\infty}v(x)>0,\,\, v(x)=0\,\, {\rm for}\,\, x\in\R^N\,\ {\rm with}\,\ x\cdot\xi\gg 1\},
$$
$$
C^+(\xi)=\{u\in X_1^+\,|\, \inf_{|x\cdot\xi|<r}u(x)>0\,\ \text{for some}\,\ r>0,\,\,u(x)=0\,\, {\rm for}\,\, x\in\R^N\,\ {\rm with}\,\ |x\cdot\xi|\gg 1\},
$$
and
$$
C^{+,1}(\xi)=\{v\in  X_2^+\,|\, \inf_{|x\cdot\xi|<r}v(x)>0\,\ \text{for some}\,\ r>0,\,\,v(x)=0\,\, {\rm for}\,\, x\in\R^N\,\ {\rm with}\,\ |x\cdot\xi|\gg 1\}.
$$

The main results of this paper are then stated in the following theorems.

\begin{tm}\label{spreading-thm}
Suppose that $b>\frac{N\mu\chi}{4}$. For any $(u_{0}, v_{0})\in C_{cp}^{+}\times C_{cp}^{+,1}$, the following hold.
 \begin{itemize}
 \item[(1)]
 For any $0<\epsilon<\sqrt{a}$,
$$
\liminf_{t\to\infty}\inf_{|x|\leq (2\sqrt{a}-\epsilon)t}u(x,t;u_0,v_0)>0,
$$
and
$$
\liminf_{t\to\infty}\inf_{|x|\leq (2\sqrt{a}-\epsilon)t}v(x,t;u_0,v_0)>0.
$$

\item[(2)] For any $\varepsilon>0$,
$$
\lim_{t\to\infty}\sup_{|x|\geq (2\sqrt{a}+\epsilon)t}u(x,t;u_0,v_0)=0,
$$
and
$$
\lim_{t\to\infty}\sup_{|x|\geq (2\sqrt{a}+\epsilon)t}v(x,t;u_0,v_0)=0.
$$
\end{itemize}
\end{tm}

\begin{tm}\label{spreading-thm-1}
Suppose that $b>\frac{N\mu\chi}{4}$. For any given $\xi\in S^{N-1}$ and $(u_{0}, v_0)\in C_{fl}^+(\xi) \times C_{fl}^{+,1}(\xi)$, the following hold.
 \begin{itemize}
 \item[(1)]
 For any $0<\epsilon<\sqrt{a}$,
$$
\liminf_{t\to\infty}\inf_{x\cdot\xi\leq (2\sqrt{a}-\epsilon)t}u(x,t;u_0,v_0)>0.
$$
and
$$
\liminf_{t\to\infty}\inf_{x\cdot\xi\leq (2\sqrt{a}-\epsilon)t}v(x,t;u_0,v_0)>0.
$$

\item[(2)] For any $\epsilon>0$,
$$
\lim_{t\to\infty}\sup_{x\cdot\xi\geq (2\sqrt{a}+\epsilon)t}u(x,t;u_0,v_0)=0,
$$
and
$$
\lim_{t\to\infty}\sup_{x\cdot\xi\geq (2\sqrt{a}+\epsilon)t}v(x,t;u_0,v_0)=0.
$$
\end{itemize}
\end{tm}

\begin{tm}\label{spreading-thm-2}
Suppose that $b>\frac{N\mu\chi}{4}$. For any given $\xi\in S^{N-1}$ and $(u_{0}, v_0)\in C^+(\xi)\times C^{+,1}(\xi)$, the following hold.
 \begin{itemize}
 \item[(1)]
 For any $0<\epsilon<\sqrt{a}$,
$$
\liminf_{t\to\infty}\inf_{|x\cdot\xi|\leq (2\sqrt{a}-\epsilon)t}u(x,t;u_0,v_0)>0.
$$

and
$$
\liminf_{t\to\infty}\inf_{|x\cdot\xi|\leq (2\sqrt{a}-\epsilon)t}v(x,t;u_0,v_0)>0.
$$

\item[(2)] For any $\epsilon>0$,
$$
\lim_{t\to\infty}\sup_{|x\cdot\xi|\geq (2\sqrt{a}+\epsilon)t}u(x,t;u_0,v_0)=0,
$$
and
$$
\lim_{t\to\infty}\sup_{|x\cdot\xi|\geq (2\sqrt{a}+\epsilon)t}v(x,t;u_0,v_0)=0.
$$
\end{itemize}
\end{tm}

We conclude the introduction with the following remarks.

\begin{rk}
\begin{itemize}
\item[(1)]  As it is recalled in the above, in the absence of chemotaxis (i.e. $\chi=0$), $2\sqrt{a}$ is the spreading speed of \eqref{fisher-kpp}. Theorems \ref{spreading-thm}, \ref{spreading-thm-1}, and \ref{spreading-thm-2} provide some new approach to prove that $2\sqrt a$ is the spreading speed of the Fisher-KPP equation \eqref{fisher-kpp}.  The new approach can also be applied to the study of the spreading speeds of \eqref{general-fisher-eq} with general time and space dependence.

\item[(2)]  Assume $b>\frac{N\mu\chi}{4}$.
 Theorem \ref{spreading-thm} (1),  Theorem \ref{spreading-thm-1} (1) and Theorem \ref{spreading-thm-2} (1) show that the chemotaxis does not slow down the spreading speed in the Fisher-KPP equation \eqref{fisher-kpp}. Theorem \ref{spreading-thm} (2),  Theorem \ref{spreading-thm-1}
(2) and Theorem \ref{spreading-thm-2} (2) show that the chemotaxis does not speed up the spreading speed in the Fisher-KPP equation \eqref{fisher-kpp}. Biologically, the condition $b>\frac{N\mu\chi}{4}$ means that the logistic damping is large relative to the product of the chemotaxis sensitivity and the production rate of the chemical substance.

\item[(3)]
Consider the following parabolic-elliptic  counterpart of \eqref{Main-Eq-1},
\begin{equation}\label{Keller-Segel-eq0-1}
\begin{cases}
u_t=\Delta u-\chi\nabla \cdot (u \nabla v)+u(a-bu),\quad x\in\Omega,\cr
0 =\Delta v-  \lambda v +\mu u,\quad x\in\Omega,\cr
\frac{\p u}{\p n}=\frac{\p v}{\p n}=0,\quad x\in\p \Omega.
\end{cases}
\end{equation}

The dynamics of \eqref{Keller-Segel-eq0-1} has been studied in many research papers and very rich dynamical scenarios have been observed. For example,
 when $a\equiv b \equiv0$, finite-time blow-up  may occur in \eqref{Keller-Segel-eq0-1} if either $N=2$ and the total initial population mass is large enough, or $N\geq 3$ (see \cite{HeMeVe, JaLu, Nagai2, Nagai3}, etc.).
When  $a$ and $b$ are positive constants,  if either $N\le 2$ or $b>\frac{N-2}{N}\chi$, then for any
nonnegative initial data $u_0\in C(\bar\Omega)$, \eqref{Keller-Segel-eq0-1} possesses a  unique bounded global classical solution $(u(x,t;u_0),v(x,t;u_0))$ with
 $u(x,0;u_0)=u_0(x)$, and hence the finite-time blow-up phenomena in \eqref{Keller-Segel-eq0-1} is suppressed  to some  extent.
Moreover, if  $b>2\chi$, then $(\frac{a}{b},\frac{\mu a}{\lambda b})$ is the unique positive steady-state solution
of  \eqref{Keller-Segel-eq0-1}, and for any nonnegative initial distribution $u_0\in C(\bar\Omega)$ ($u_0(x)\not \equiv 0$),
 $$
 \lim_{t\to\infty} \big[ \|u(\cdot,t;u_0)-\frac{a}{b}\|_{L^\infty(\Omega)}+\|v(\cdot,t;u_0)-\frac{\mu a}{\lambda b}\|_{L^\infty(\Omega)}\big]=0
 $$
 (hence the chemotaxis does not affect the limiting distribution).
 But if $b<2\chi$, there may be more than one positive steady-state solutions of \eqref{Keller-Segel-eq0-1} (see \cite{TeWi}).

\item[(4)]
Consider the following parabolic-elliptic counterpart of \eqref{Main-Eq},
\begin{equation}\label{Keller-Segel-eq0}
\begin{cases}
u_t=\Delta u-\chi\nabla \cdot (u \nabla v)+u(a-bu),\quad x\in\R^{N},\cr
0 =\Delta v-  \lambda v +\mu u,\quad x\in\R^{N}.
\end{cases}
\end{equation}
The authors of  \cite{SaSh1, SaSh2, SaSh3-1, SaShXu} studied the spatial spreading dynamics of \eqref{Keller-Segel-eq0}.
Among others, it is proved that,  if $b>\chi\mu$ and $b\geq \big(1+\frac{1}{2}\frac{(\sqrt{a}-\sqrt{\lambda})_+}{(\sqrt{a}+\sqrt{\la})}\big)\chi\mu$,
$c_0^*:=2\sqrt a$ is the spreading speed of the solutions of  \eqref{Keller-Segel-eq0} with nonnegative continuous initial function $u_0$ with nonempty compact support, that is,
 $$
 \lim_{t\to\infty}\sup_{|x|\ge ct}u(x, t;u_0)=0\quad \forall\,\, c>c_0^*
 $$
 and
 $$
 \lim_{t\to\infty}\inf_{|x|\le ct} u(x, t;u_0)>0\quad \forall \,\, 0<c<c_0^*,
 $$
 where $(u(x, t;u_0),v(x, t;u_0))$ is the unique global classical solution of \eqref{Keller-Segel-eq0} with $u(x, 0;u_0)=u_0(x)$.
 It is also proved that,  if  $b>2\chi\mu$  and $\lambda \geq a$ hold, then  $2\sqrt a$ is the minimal speed  of the  traveling wave solutions of \eqref{Keller-Segel-eq0}  connecting $(0,0)$ and $(\frac{a}{b},\frac{\mu}{\lambda}\frac{a}{b})$, that is,
 for any $c\ge 2\sqrt a$, \eqref{Keller-Segel-eq0}
 has a  traveling wave solution  connecting $(0,0)$ and $(\frac{a}{b},\frac{\mu}{\lambda}\frac{a}{b})$ with speed $c$,
 and \eqref{Keller-Segel-eq0} has no such traveling wave solutions with speed less than $2\sqrt a$. In particular, if $\lambda \geq a$ and $b>\chi\mu$, or $\lambda<a$ and  $b\geq \big(1+\frac{1}{2}\frac{(\sqrt{a}-\sqrt{\lambda})}{(\sqrt{a}+\sqrt{\la})}\big)\chi\mu$, then the chemotaxis neither speeds up nor slows down the spatial spreading in {the Fisher-KPP equation \eqref{fisher-kpp}}.

\end{itemize}
\end{rk}

The rest of the paper is organized as follows: In section 2, we study the lower bounds of spreading speeds of \eqref{Main-Eq} and prove Theorem \ref{spreading-thm} (1), Theorem \ref{spreading-thm-1} (1) and Theorem \ref{spreading-thm-2} (1). In section 3, we explore the upper bounds of spreading speeds of \eqref{Main-Eq} and prove Theorem \ref{spreading-thm} (2), Theorem \ref{spreading-thm-1} (2) and Theorem \ref{spreading-thm-2} (2).

\section{Lower bounds of spreading speeds}

In this section, we investigate lower bounds of spreading speeds of  global classical solutions of \eqref{Main-Eq} with different initial functions. We first prove some preliminary lemmas in subsection 2.1. Then we  prove Theorem \ref{spreading-thm} (1), Theorem \ref{spreading-thm-1} (1),  and Theorem \ref{spreading-thm-2} (1) in subsections 2.2, 2.3, and 2.4, respectively.
Throughout this section, we assume that $b>\frac{N\mu\chi}{4}$.

\subsection{Preliminary lemmas}

In this subsection, we present some lemmas to be used in the proofs of Theorem \ref{spreading-thm} (1), Theorem \ref{spreading-thm-1} (1),  and Theorem \ref{spreading-thm-2} (1).

For any given $\xi\in S^{N-1}$ and  $c\in\R$, let $\tilde u(x,t)=u(x+ct\xi, t)$ and $\tilde v(x,t)=v(x+ct\xi, t)$. Then \eqref{Main-Eq} becomes
\begin{equation}\label{tilde-u-v-eq}
\begin{cases}
\tilde u_{t}=\Delta\tilde u+c\xi\cdot\nabla \tilde u- \chi \nabla\cdot (\tilde u\nabla \tilde v) + \tilde u(a-b\tilde u)\quad  x\in\R^N, \\
\tilde v_{t}=\Delta \tilde v+c\xi\cdot\nabla\tilde v -\lambda \tilde v+\mu \tilde u,\quad x\in\R^N.
\end{cases}
\end{equation}
In the following, $(\tilde u(x,t;\xi, c, u_0,v_0), \tilde v(x,t;\xi, c, u_0,v_0))$ denotes the classical solution of \eqref{tilde-u-v-eq} with $\tilde u(x,0;\xi, c, u_0,v_0)=u_0\in X_1^+$ and $\tilde v(x,0;\xi, c, u_0,v_0))=v_0\in  X_2^+$.

For any given $0<\epsilon<\sqrt{a}$, fix $0<\bar a<a$  such that
\begin{equation}
\label{bar-r-eq}
4\bar a -c^2\ge \epsilon \sqrt{a} \quad \forall\, -2\sqrt{a}+\epsilon \le  c  \le  2\sqrt {a}-\epsilon.
\end{equation}
 Let
\begin{equation}
\label{L-eq}
l=\frac{2\pi\sqrt{N}}{(\epsilon\sqrt{a})^{\frac{1}{2}}}
\end{equation}
and
\begin{equation}
\label{lambda-eq}
\lambda(c,\bar a)=\frac{4\bar a-c^2-\frac{N\pi^2}{l^2}}{4}.
\end{equation}
 Then $\lambda(c,\bar a)\ge  \frac{3\epsilon\sqrt {a}}{16}>0$  for any $ -2\sqrt{a}+\epsilon  \le c \le   2\sqrt {a}-\epsilon$.
Let
$$
D_{l}=\{x\in\R^N \,\ | \,\ |x_{i}|<l\,\ {\rm for}\,\ i=1,2,\cdot\cdot\cdot N\}.
$$
For every $x\in\R^N$, and $r>0$, we define
$$
B_{r}(x):=\{y\in\R^N \,|\, |y-x|<r\}.
$$

\begin{lem}
\label{lem-001-2}
For any given  $0<\epsilon<\sqrt{a}$, let $\bar a$ and $l$ be as in \eqref{bar-r-eq} and
\eqref{L-eq}. Then
for any $ -2\sqrt{a}+\epsilon \leq c \leq  2\sqrt {a}-\epsilon$ and $\xi\in S^{N-1}$, $\lambda(c,\bar a)$ which is defined as in \eqref{lambda-eq} is the principal eigenvalue of
\begin{equation*}
\label{ev-eq1}
\begin{cases}
\Delta\phi+ c\xi\cdot\nabla\phi+\bar a \phi=\lambda \phi,\quad x\in D_{l}\cr
\phi(x)=0, \quad x\in \partial D_{l},
\end{cases}
\end{equation*}
and $\phi(x;\xi, c,\bar a)=e^{-\frac{c}{2}\xi\cdot x}\prod_{i=1}^{N}{ \cos\frac{\pi}{2l}x_{i}}$ is a corresponding positive eigenfunction.
\end{lem}

\begin{proof}
It follows from direct calculations.
\end{proof}

\begin{lem}\label{lem-2.2}
\label{persistence-lm}
 There are  $M>0$, $M_1>0$, and $0<\theta<\frac{1}{2}$ such that for any $(u_0,v_0)\in X_1^+\times X_2^+$, there is
$T_0(u_0,v_0)>1$ such that for any $c\in\R$, any $\xi\in S^{N-1}$, it holds that
\begin{equation*}
\label{condition-on-initial-eq}
\begin{cases}
\|\tilde u(\cdot,t;\xi, c,u_0,v_0)\|_\infty\le M\quad \forall \, t\ge T_0(u_0,v_0)\cr
 \|\tilde v(\cdot,t;\xi, c,u_0,v_0)\|_\infty\le M\quad \forall\, t\ge T_0(u_0,v_0)\cr
\|\nabla\tilde v(\cdot,t;\xi, c,u_0,v_0)\|_\infty\le M\quad \forall\, t\ge T_0(u_0,v_0)\cr
 \|\Delta\tilde v(\cdot,t;\xi, c,u_0,v_0)\|_\infty\le M\quad \forall\, t\ge T_0(u_0,v_0)
 \end{cases}
 \end{equation*}
 and
 \begin{equation*}
 \label{condition-on-initial-eq-0}
 \sup_{t,s\ge T_0(u_0,v_0)+1, t\not= s} \frac{\|\nabla\tilde v(\cdot,t;c, u_0,v_0)-\nabla\tilde v(\cdot,s;c, u_0,v_0)\|_\infty}{|t-s|^\theta}
 \le M M_1.
\end{equation*}
 \end{lem}

\begin{proof}
It follows from \cite[Lemma 4.1]{ShXu}.
\end{proof}

In the following,
 $M>0$ is as in Lemma \ref{lem-2.2}, and for given $0<\epsilon<\sqrt a$, $l>0$ is as in  \eqref{L-eq}.
For given $\eta>0$, let $T=T(\eta)\ge 1$ be such that
\begin{equation}\label{T-choice}
 e^{-\lambda T}M\leq \eta,
\end{equation}
 and $L=L(\eta)\geq l$ be  such that $B_{L}(0)\supset D_{l}$ and
\begin{equation}\label{L-choice}
\max\{\int_{\R^{N}\backslash {B}_{\frac{L-4T\sqrt{a}}{2\sqrt{2T}}}(0)}e^{-|z|^2}dz, \int_{\R^{N}\backslash {B}_{\frac{L-4T\sqrt{a}}{2\sqrt{2T}}}(0)}|z|e^{-|z|^2}dz\}\leq \eta.
\end{equation}

\begin{lem}\label{lem-1}
For any given  $0<\epsilon<\sqrt{a}$, let $\bar a$ and $l$ be as in \eqref{bar-r-eq} and
\eqref{L-eq}.
Let $0<\tilde a<a-\bar a$ be fixed.
There is $\epsilon_0>0$ such that for any $0<\eta\le \epsilon_0$,  any $(u_{0}, v_{0})\in X_{1}^{+} \times X_{2}^{+}$,  any $\xi\in S^{N-1}$, any $-2\sqrt{a}+\epsilon \le  c  \le  2\sqrt {a}-\epsilon$, any $t_1,t_2$ satisfying $T_0(u_0,v_0)\le t_1< t_2\le  \infty$, and any ball $B_{2L(\eta)}$ with radius $2L(\eta)$ in $\R^{N}$,
if
$$
\sup_{x\in B_{2L(\eta)}} \tilde u(x,t;\xi,c, u_0,v_0)\le \eta\quad \forall\, t_1\le t< t_2,
$$
then
\begin{equation}
\label{v-dv-est}
\sup_{x\in B_{L(\eta)}}\max\{\tilde v(x,t;\xi,c, u_0,v_0),|\p_{x_{i}}\tilde v(x,t;\xi,c, u_0,v_0)|\}\leq \tilde M\eta\quad \forall\, t_1+T(\eta)\le t< t_2
\end{equation}
 and
 \begin{equation}
 \label{ddv-dtv-est}
 \chi \sup_{x\in B_{L(\eta)}} \sum_{i,j=1}^{N}|\p_{x_{i}x_{j}} \tilde v(x,t;\xi,c, u_0,v_0)|\le \tilde a\quad \forall\, t_1+T(\eta)+1\le t<t_2,
 \end{equation}
 where
 $$
 \tilde M=\max\big \{  1+\frac{\mu M}{\lambda \pi^{\frac{N}{2}}}+\frac{\mu}{\lambda},\,\,  1+\frac{ \mu}{\pi^{\frac{N}{2}}}\lambda^{-\frac{1}{2}}\Gamma(\frac{1}{2})M+\frac{ \mu}{\pi^{\frac{N}{2}}}\lambda^{-\frac{1}{2}}\Gamma(\frac{1}{2})\big \}.
 $$
\end{lem}

\begin{proof}
It suffices to prove the lemma for the ball centered at the origin with radius $2L(\eta)$. If not, we can make appropriate translation of $(\tilde u(x,t;\xi,c, u_0,v_0), \tilde v(x,t;\xi,c, u_0,v_0))$ for the space variable $x$ to achieve this.
We first prove that \eqref{v-dv-est} holds for any $\eta>0$.
Fix $t_1\geq T_0(u_0,v_0)$. Note that
 \begin{align*}
&\tilde v(x,t;\xi,c, u_0,v_0)\\
&=\int_{\R^{N}}\frac{e^{-\lambda (t-t_1)}}{(4\pi (t-t_1))^{\frac{N}{2}}}e^{-\frac{|x+c(t-t_1)\xi-y|^{2}}{4(t-t_1)}}\tilde v(y, t_1;\xi, c, u_0,v_0)dy \\
&\,\,+\mu \int_{t_1}^{t}\int_{\R^{N}}\frac{e^{-\lambda (t-s)}}{(4\pi (t-s))^{\frac{N}{2}}}e^{-\frac{|x+c(t-s)\xi-y|^{2}}{4(t-s)}}\tilde u(y, s;\xi, c, u_0,v_0)dyds\\
&=\frac{1}{\pi^{\frac{N}{2}}}\int_{\R^{N}}e^{-\lambda (t-t_1)}e^{-|z|^2}\tilde v(x+c(t-t_1)\xi+2\sqrt{t-t_1}z, t_1;\xi, c,u_0,v_0)dz \\
&\,\,+\frac{\mu}{\pi^{\frac{N}{2}}} \int_{t_1}^{t}\int_{\R^{N}}e^{-\lambda (t-s)}e^{-|z|^2}\tilde u(x+c(t-s)\xi+2\sqrt{t-s}z, s;\xi,c,u_0,v_0)dzds,
\end{align*}
and
\begin{align*}\label{v-deriv-est}
&\partial_{x_{i}}\tilde v(x,t;\xi,c, u_0,v_0)\\
&=\int_{\R^N}\frac{(y_{i}-x_{i}-c(t-t_1)\xi)e^{-\lambda (t-t_1)}}{2(t-t_1)(4\pi (t-t_1))^{\frac{N}{2}}}e^{-\frac{|x+c(t-t_1)\xi-y|^{2}}{4(t-t_1)}}\tilde v(y, t_1;\xi,c,u_0,v_0)dy\\
&\,\,+\mu \int_{t_1}^{t}\int_{\R^N}\frac{(y_{i}-x_{i}-c(t-s)\xi)e^{-\lambda (t-s)}}{2(t-s)(4\pi (t-s))^{\frac{N}{2}}}e^{-\frac{|x+c(t-s)\xi-y|^{2}}{4(t-s)}}\tilde u(y, s;\xi,c,u_0,v_0)dyds\\
&=\frac{1}{\pi^{\frac{N}{2}}}(t-t_1)^{-\frac{1}{2}}e^{-\lambda (t-t_1)}\int_{\R^N}ze^{-z^2}\tilde v(x+c(t-t_1)\xi+2\sqrt{t-t_1}z, t_1;\xi,c,u_0,v_0)dz\\
&\,\,+\frac{\mu}{\pi^{\frac{N}{2}}} \int_{t_1}^{t}\int_{\R^N}(t-s)^{-\frac{1}{2}}e^{-\lambda (t-s)}ze^{-z^2}\tilde u(x+c(t-s)\xi+2\sqrt{t-s}z, s;\xi,c,u_0,v_0)dzds.
\end{align*}
Hence, for $x\in B_{L}(0)$ and $t_1+T\leq t\leq \min\{t_1+2T,t_2\}$, we have
\begin{align*}
\tilde v(x,t; \xi,c, u_0,v_0)&\leq e^{-\lambda T}M+\frac{ \mu}{\pi^{\frac{N}{2}}} \left[\int_{t_1}^{t}\int_{\R^N\backslash {B}_{\frac{L-4T\sqrt{a}}{2\sqrt{2T}}}(0)}e^{-\lambda (t-s)}e^{-|z|^2}dzds\right]M \cr
&+\frac{ \mu}{\pi^{\frac{N}{2}}} \left[\int_{t_1}^{t}\int_{{B}_{ \frac{L-4T\sqrt{a}}{2\sqrt{2T}}}(0)}e^{-\lambda (t-s)}e^{-|z|^2}dzds\right]\sup_{t_1\leq t< t_2, |z|\leq 2L}\tilde u(z, t;\xi,c,u_0,v_0).
\end{align*}
By \eqref{T-choice} and \eqref{L-choice},  if $\sup_{x\in B_{2L}(0)} \tilde u(x,t;\xi,c, u_0,v_0)\le \eta$ for any $t_1\le t< t_2$, then
\begin{equation}\label{v-est}
\tilde v(x,t;\xi,c, u_0,v_0)\leq (1+\frac{\mu M}{\lambda \pi^{\frac{N}{2}}}+\frac{\mu}{\lambda})\eta \quad \forall\,\ t_1+T\leq t\leq \min\{t_1+2T,t_2\},\,\
|x|\leq L.
\end{equation}
For $t_1+T\leq t\leq \min\{t_1+2T,t_2\}$, and  $x\in B_{L}(0)$, we have
\begin{align*}\label{partial-v-est}
&|\partial_{x_{i}}\tilde v(x,t;\xi,c, u_0,v_0)|\cr
& \leq \frac{1}{\pi^{\frac{N}{2}}}T^{-\frac{1}{2}}e^{-\lambda T}M+\frac{ \mu}{\pi^{\frac{N}{2}}} \left[\int_{t_1}^{t}\int_{\R^{N}\backslash {B}_{\frac{L-4T\sqrt{a}}{2\sqrt{2T}}}(0)}(t-s)^{-\frac{1}{2}}e^{-\lambda (t-s)}|z|e^{-|z|^2}dzds\right]M \cr
&+\frac{ \mu}{\pi^{\frac{N}{2}}} \left[\int_{t_1}^{t}\int_{{B}_{ \frac{L-4T\sqrt{a}}{2\sqrt{2T}}}(0)}(t-s)^{-\frac{1}{2}}e^{-\lambda (t-s)}|z|e^{-|z|^2}dzds\right]\sup_{t_1\leq t< t_2, |z|\leq 2L}\tilde u(z,t;\xi,c,u_0,v_0).
\end{align*}
By \eqref{T-choice} and \eqref{L-choice},  if $\sup_{x\in B_{2L}(0)} \tilde u(x,t;\xi,c, u_0,v_0)\le \eta$ for any $t_1\le t< t_2$, then
\begin{equation}
\label{partial-v-est-1}
|\partial_{x_{i}}\tilde v(x,t;\xi,c,u_0,v_0)|\leq (1+\frac{ \mu}{\pi^{\frac{N}{2}}}\lambda^{-\frac{1}{2}}\Gamma(\frac{1}{2})M+\frac{ \mu}{\pi^{\frac{N}{2}}}\lambda^{-\frac{1}{2}}\Gamma(\frac{1}{2}))\eta
 \end{equation}
 for  $t_1+T\leq t\le \min\{t_1+2T, t_2\}$ and $x\in B_{L}(0)$.

 In the above arguments, replace $t_1$ by $t_1+T$. We have \eqref{v-est} and \eqref{partial-v-est-1}
 for $t_1+2T\le t\le \min\{t_1+3T,t_2\}$. Repeating this process, we have \eqref{v-est} and \eqref{partial-v-est-1}
 for $t_1+T\le t<t_2$.
It then follows that \eqref{v-dv-est} holds for any $\eta>0$.

Next, we prove that there is $\epsilon_0>0$ such that \eqref{ddv-dtv-est} holds for $0<\eta\le \epsilon_0$. Assume this is  not true. Then there are $\eta_n\to 0$ as $n\to\infty$,
$(u_n,v_n)\in X_{1}^{+} \times X_{2}^{+}$, $\xi_{n}\in S^{N-1}$, $-2\sqrt{a}+\epsilon \le  c_n  \le  2\sqrt {a}-\epsilon$, $T_0(u_n,v_n)\le t_{1n}<t_{1n}+T(\eta_n)+1\le t_n<t_{2n}$ such that
$$
\sup_{|x|\leq 2L(\eta_{n})}\tilde u(x,t;\xi_{n}, c_n, u_n,v_n)\le \eta_n, \quad \forall\, t_{1n}\le t< t_{2n}
$$
and
$$
\chi \sup_{|x|\leq L(\eta_{n})}\sum_{i,j=1}^{N} |\p_{x_{i}x_{j}}\tilde v(x,t_n;\xi_{n}, c_n, u_n,v_n)|> \tilde a.
$$
Let
$$
(\tilde u_n(x,t),\tilde v_n(x,t))=(\tilde u(x,t+t_n;\xi_{n}, c_n, u_n,v_n),\tilde v(x,t+t_n;\xi_{n}, c_n, u_n,v_n)).
$$
Without loss of generality, we may assume that
$$
(\tilde u_n(x,t),\tilde v_n(x,t))\to (u^*(x,t),v^*(x,t))
$$
as $n\to \infty$ locally uniformly on $(x,t)\in \R^{N}\times [-1,\infty)$, $\xi_{n}\to \xi^{*}$, $c_{n}\to c^*$ as $n\to\infty$ for some $\xi^{*}\in S^{N-1}$, $-2\sqrt{a}+\epsilon \le  c^*  \le  2\sqrt {a}-\epsilon$. Note that $v^*(x,t)$ satisfies
$$
v_t^*= \Delta v^{*}+c^{*}\xi^{*}\cdot \nabla v^{*}-\lambda v^*+\mu u^*,\quad \forall\,\ x\in\R^{N},\,\, t\ge -1
$$
and
$$
\chi \sup_{x\in\R^{N}}\sum_{i,j=1}^{N}|\p_{x_{i}x_{j}}v^*(x,0)|\geq \tilde a.
$$
By \eqref{v-dv-est}, we have
$$
u^*(x,t)=0,\,\, v^*(x,t)=0\quad \forall\,\ {x\in \R^N},\,\, -1\le t \le 0.
$$
Then by the comparison principle for parabolic equations,
$$
v^*(x,t)= 0\quad\forall\, x\in\R^{N},\,\, t\ge -1,
$$
which is a contradiction. Hence  \eqref{ddv-dtv-est} holds.
\end{proof}

\begin{lem}\label{lem-2} For any given  $0<\epsilon<\sqrt{a}$, let $\bar a$ and $l$ be as in \eqref{bar-r-eq} and
\eqref{L-eq}.
Let $\lambda_0=\min_{-2\sqrt{a}+\epsilon\leq c\leq 2\sqrt{a}-\epsilon}\lambda(c,\bar a)>0$, where $\lambda(c,\bar a)$ is as in Lemma \ref{lem-001-2}. Let $\tilde T_0\geq 1$ be such that
$e^{\lambda_{0}\tilde T_0}\ge 4$. Let $\epsilon_0$ be as in Lemma \ref{lem-1}.
For any $0<\eta\le \epsilon_0$, there is $0<\delta_{\eta}\leq \epsilon_0$ such that for any $(u_0,v_0)\in X_{1}^+ \times X_{2}^{+}$, any $\xi\in S^{N-1}$, any $-2\sqrt{a}+\epsilon\leq c\leq 2\sqrt{a}-\epsilon$, any  $t_0\ge T_0(u_0,v_0)+2$, and any ball $B_{2L}\subset\R^{N}$ with radius $2L$, if
$$
\sup_{x\in B_{2L}} \tilde u(x, t_0;\xi,c,u_0,v_0)\ge \eta,
$$
then
$$
\inf_{x\in B_{2L}} \tilde u(x,t;\xi,c,u_0,v_0)\ge \delta_{\eta}\quad \forall\, t_0\le t\le t_0+T+\tilde T_0,
$$
where $L=L(\eta)$ and $T=T(\eta)$.

\end{lem}

\begin{proof}
Suppose on the contrary that the conclusion fails. Then there exist $0<\eta_0\le \epsilon_0$,
$(u_{0n},v_{0n})\in X_{1}^+ \times X_{2}^{+}$, $\xi_{n}\in S^{N-1}$, $-2\sqrt{a}+\epsilon\leq c_{n}\leq 2\sqrt{a}-\epsilon$, $t_{0n}\geq T_0(u_{0n},v_{0n})+2$, a sequence of ball $B^{n}_{2L(\eta_0)}\subset \R^{N}$ with radius $2L(\eta_{0})$, $x_n$, $x^*_{n}\in\R$ with $x_n\in B^{n}_{2L(\eta_0)}$, $x^*_{n}\in B^{n}_{2L(\eta_0)}$, $t_{n}\in\R$ with $t_{0n}\leq t_{n}\leq t_{0n}+T(\eta_0)+ \tilde T_0$ such that
\begin{equation}\label{u-t-0-n}
\lim_{n\to\infty} \tilde u(x_n, t_{0n};\xi_{n}, c_n, u_{0n}, v_{0n})\geq \eta_0
\end{equation}
and
\begin{equation}\label{u-t-n}
\lim_{n\to\infty} \tilde u(x^*_{n}, t_{n};\xi_{n}, c_n, u_{0n}, v_{0n})=0.
\end{equation}
Let $\tilde u_n(x,t)=\tilde u(x+x_{n} ,t+t_{0n}-1;\xi_{n},c_n, u_{0n},v_{0n})$, $\tilde v_n(x,t)=\tilde v(x+x_{n}, t+t_{0n}-1;\xi_{n}, c_n, u_{0n},v_{0n})$, and
$T=T(\eta_0)+ \tilde T_0$, $L=L(\eta_0)$.
Without loss of generality,
we may assume that
\begin{equation*}\label{x-t-conv}
x^{*}_{n}-x_{n}\to x^{*},  \quad t_{n}-t_{0n}+1\to t^{*}\geq 1 \quad {\rm as}\quad  n\to\infty
\end{equation*}
and
\begin{equation*}\label{u-n-conv}
(\tilde u_n(x,t),\tilde v_n(x,t))\to (u^*(x,t),v^*(x,t))
\end{equation*}
as $n\to\infty$ locally uniformly in $(x,t)\in \R^{N}\times [0,\infty)$, $\xi_{n}\to \xi^{*}$ and $c_n\to c^*$ as $n\to\infty$ for some $\xi^{*}\in S^{N-1}$, $-2\sqrt{a}+\epsilon\leq c^{*}\leq 2\sqrt{a}-\epsilon$. Then  $(u^*, v^*)$ is a solution of \eqref{tilde-u-v-eq} with $\xi$ being replaced by $\xi^{*}$ and $c$ being replaced by $c^{*}$ for $t\ge 0$.

By \eqref{u-t-0-n}, $u^{*}(0,1)\geq \eta_{0}$, it follows from comparison principle for parabolic equations that $u^{*}(x,t)>0$ for $x\in\R^{N}$, $t>0$. But by \eqref{u-t-n}, $u^{*}(x^{*}, t^{*})=0$. This is a contraction.
\end{proof}

\begin{lem}\label{lem-3}
For any given  $0<\epsilon<\sqrt{a}$, let $\bar a$ and $l$ be as in \eqref{bar-r-eq} and
\eqref{L-eq}.
There is $ 0<\tilde\epsilon_0\le \epsilon_0$ such that for any $0<\eta\le  \tilde\epsilon_0$, there is $\tilde\delta_\eta>0$ such that
for any $(u_0,v_0)\in X_{1}^+ \times X_{2}^{+}$, any $\xi\in S^{N-1}$, any $-2\sqrt{a}+\epsilon\leq c\leq 2\sqrt{a}-\epsilon$, any $t_1,t_2$ satisfying that $T_0(u_0,v_0){+2}\le t_1< t_2\le \infty$, and any ball $B_{2L}\subset\R^{N}$ with radius $2L$,
if
$$
\sup_{x\in B_{2L}}\tilde u(x, t_1;\xi, c, u_0,v_0)=\eta,\, \,\, \sup_{x\in B_{2L}} \tilde u(x,t;\xi, c, u_0,v_0)\le \eta,\,\, \forall\,\ t_1<t<t_2,
$$
then
$$
\inf_{x\in B_{2L}} \tilde u(x,t;\xi, c, u_0,v_0)\ge \tilde\delta_\eta\quad \forall\,\ t_1\le t< t_2,
$$
where $L=L(\eta)$.
\end{lem}

\begin{proof}
It suffices to prove the lemma for the ball $B_{2L}(0)$ centered at the origin with the radius $2L$. If the ball with the radius $2L$ is not centered at the origin, we can make an appropriate translation of $(\tilde u(x, t;\xi, c, u_0,v_0), \tilde v(x, t;\xi, c, u_0,v_0))$ for the space variable $x$ to shift the ball into the ball centered at the origin.

 First, consider
\begin{equation}
\label{new-new-eqq1}
\begin{cases}
 u_t=\Delta u+c\xi \cdot \nabla u+q(x,t)\cdot \nabla u+\bar au, \quad x\in D_{l}, \,\, t>0\\
 u(x,t)=0, \quad  x\in \p{D_{l}}, \,\,  t>0,\\
 u(x,0)=\bar \phi(x;\xi,c,\bar a), \quad x\in D_{l},
\end{cases}
\end{equation}
where $\bar \phi (x;\xi,c,\bar a)=\frac{\phi(x;\xi,c,\bar a)}{\|\phi\|_\infty}$ and  $\phi(x;\xi, c,\bar a)$ is as in Lemma \ref{lem-001-2}.
Let $ \bar u(x,t;\xi,c,q)$ be the solution of \eqref{new-new-eqq1}. Let $ \tilde T_0\geq 1$ be as in lemma \ref{lem-2}.
We claim that there is $\tilde\epsilon_0>0$ such that for any $-2\sqrt{a}+\epsilon\leq c\leq 2\sqrt{a}-\epsilon$, any $\xi\in S^{N-1}$, any function $q(x,t)$ which is $C^1$ in $x$ {and  H\"older continuous in $t$} with exponent $0<\theta<\frac{1}{2}$,
\begin{equation}
\label{q-eq1}
\sup_{t\ge 0}\|q(\cdot,t)\|_{C(\bar D_{l})}\le
\chi \sqrt{N}\tilde M \tilde\epsilon_0
 \end{equation}
 ($\tilde M$ is as in Lemma \ref{lem-1}),  and
 \begin{equation}
 \label{q-eq2}
 \sup_{{t,s\ge 0, t\not =s}}\frac{\|q(\cdot,t)-q(\cdot,s)\|_{{C(\bar D_{l})}}}{|t-s|^\theta}\le \chi M M_1
 \end{equation}
  ($M$ and $M_1$ are as in Lemma \ref{persistence-lm}), there holds
\begin{equation}
\label{new-new-eqq2}
 \bar u(x, \tilde T_0;\xi,c,q)\ge 2 \bar \phi(x;\xi, c,\bar a)\quad \forall\,\ x\in D_{l}.
\end{equation}
In fact, assume this is not true. Then there are $\epsilon_n\to 0$ as $n\to\infty$, $x_n\in D_{l}$, $\xi_{n}\in S^{N-1}$, $-2\sqrt{a}+\epsilon\leq c_{n}\leq 2\sqrt{a}-\epsilon$,  and $q_n(x,t)$  satisfying \eqref{q-eq2} and
$$
\sup_{t\ge 0}\|q_{n}(\cdot,t)\|_{C(\bar D_{l})}\le
\chi \sqrt{N}\tilde M \epsilon_n
$$
such that
\begin{equation}\label{u-bar-eq}
\bar u(x_n, \tilde T_0;\xi_{n}, c_n, q_{n})< 2 \bar \phi(x_n;\xi_{n}, c_{n},\bar a) \quad \forall\,\ n\geq 1.
\end{equation}
Let $u_n(x,t)= \bar u(x,t;\xi_{n}, c_n,q_n)$. Without loss of generality, we may assume that
$$
u_n(x,t)\to u^*(x,t),\quad \p_{x_j} u_n(x,t)\to \p_{x_j} u^*(x,t)\quad  {\rm as} \quad n\to\infty
$$ locally uniformly in $(x,t)\in \bar D_{l}\times [0,\infty)$, $\xi_{n}\to \xi^{*}$ and $c_n\to c^*$ as $n\to\infty$ for some $\xi^{*}\in S^{N-1}$, $-2\sqrt{a}+\epsilon\leq c^{*}\leq 2\sqrt{a}-\epsilon$.
Note that
$u^*(x,t)= \bar u(x,t;\xi^{*}, c^{*},0)=e^{\lambda(c^{*},\bar a)t}\bar \phi(x;\xi^{*}, c^{*},\bar a)$. Hence
$$
u^*(x, \tilde T_0)\ge e^{\lambda_{0}\tilde T_{0}}\bar \phi(x;\xi^{*}, c^{*},\bar a)\ge 4\bar  \phi(x;\xi^{*}, c^{*},\bar a),\quad \forall\,\ x\in D_{l}.
$$
This together with the Hopf's Lemma implies that
$$
u_n(x, \tilde T_0)\ge 2\bar \phi(x;\xi_{n}, c_{n},\bar a)\quad \forall\,\ x\in D_{l},\,\,\, n\gg 1,
$$
which contradicts to \eqref{u-bar-eq}. Hence the claim holds true.

Next, without loss of generality, we may assume that
$$
a-\tilde a-b\tilde \epsilon_{0}\geq \bar a.
$$
Let $T=T(\eta)$.
By Lemma \ref{lem-1}, for any given
  $0<\eta\le \tilde\epsilon_0$, $\xi\in S^{N-1}$, $-2\sqrt{a}+\epsilon\leq c\leq 2\sqrt{a}-\epsilon$,
  $t_1+T+1\le t< t_2\leq \infty$, and $x\in B_{L}(0)$,
\begin{align*}\label{u-sup-eq}
\tilde u_{t}&=\Delta\tilde u+c\xi\cdot\nabla\tilde u - \chi \nabla\tilde v\cdot\nabla\tilde u+\tilde u(a-\chi \Delta\tilde v -b\tilde u)\cr
&\geq \Delta\tilde u+c\xi\cdot\nabla\tilde u+q(x,t)\cdot\nabla\tilde u+\bar a \tilde u,
\end{align*}
where $q(x,t)=- \chi  \nabla \tilde v(x,t;\xi, c, u_0,v_0)$. By Lemma \ref{persistence-lm} and Lemma \ref{lem-1}, $q(\cdot,\cdot{+t_1+T+1})$ satisfies \eqref{q-eq1} and \eqref{q-eq2}.
Let $n_0\ge 0$ be such that
  $$t_1+T+1+ n_0 \tilde T_0<t_2\quad {\rm and}\quad t_1+T+1+(n_0+1) \tilde T_0\ge t_2.
  $$
 By Lemma \ref{lem-2},
$$
\inf_{x\in B_{2L}(0)} \tilde u(x,t;\xi, c, u_0,v_0)\ge \delta_\eta\quad \forall\, t_1\le t\le t_1+T+1.
$$
This together with
 the comparison principle for parabolic equations and \eqref{new-new-eqq2} implies that for any $-2\sqrt{a}+\epsilon\leq c\leq 2\sqrt{a}-\epsilon$, any $\xi\in S^{N-1}$,
\begin{align*}
\tilde u(x,t_1+T+1+k \tilde T_0;\xi, c, u_0,v_0)&\geq  2^{k-1}\delta_\eta  \bar u(x,\tilde T_0;\xi, c,q(\cdot,\cdot+t_1+T+1+(k-1) \tilde T_0))\cr
&\geq  2^{k} \delta_{\eta}\bar  \phi(x;\xi, c, \bar a) \quad \forall\,\ x\in D_{l}
\end{align*}
for $k=1,2,\cdots,n_0$, where $\delta_\eta$ is as in Lemma \ref{lem-2}.
By Lemma \ref{lem-2} again, we then have for any $-2\sqrt{a}+\epsilon\leq c\leq 2\sqrt{a}-\epsilon$, any $\xi\in S^{N-1}$,
$$
\inf_{x\in B_{2L}(0)}\tilde u(x,t;\xi, c, u_0,v_0)\ge \tilde\delta_\eta:=\min\{\delta_\eta,\delta_{\delta_\eta}\}\quad \forall \, t_1\le t<t_2.
$$
\end{proof}

\subsection{Proof of Theorem \ref{spreading-thm} (1)}

In this subsection, we prove  Theorem \ref{spreading-thm} (1). Throughout this subsection, let $(u_{0}, v_{0})\in C_{cp}^{+}\times C_{cp}^{+,1}$ be fixed.

\begin{proof}[Proof of Theorem \ref{spreading-thm} (1)]

We first prove that
for any $0<\epsilon<\sqrt{a}$,
\begin{equation}\label{lower-cp-u-1-1}
\liminf_{t\to\infty}\inf_{|x|\leq (2\sqrt{a}-\epsilon)t}u(x,t;u_0,v_0)>0.
\end{equation}
For any $0<\epsilon<\sqrt{a}$, let $\bar a$ and $l$ be as in \eqref{bar-r-eq} and \eqref{L-eq}.
Let $T_{0}=T_{0}(u_0,v_0)$ and $\tilde \epsilon_0$ be as in Lemma \ref{lem-2.2} and Lemma \ref{lem-3}, respectively. Let $T(\tilde \epsilon_0)$ be such that \eqref{T-choice} holds.
 For any $-2\sqrt{a}+\epsilon\leq c\leq 2\sqrt{a}-\epsilon$, any $\xi\in S^{N-1}$, let
$$
\tilde \delta:=\tilde \delta(\xi, c)=\inf_{x\in \bar D_{l}} \tilde u(x,T_0+T(\tilde \epsilon_0)+3;\xi, c, u_0,v_0).
$$
By the assumption $u_0(x)\geq 0$ and $u_0(x)\not\equiv 0$,  {$\tilde \delta>0$}. Let
$$
k_0=\inf\{k\in \Z^+\,|\, 2^{k}{\tilde \delta}\ge \tilde\epsilon_0\}
\quad {\rm and}\quad
T_{00}=T_0+T(\tilde\epsilon_0){+3}+k_0  \tilde T_0,
$$
where $ \tilde T_0\geq 1$ is as in lemma \ref{lem-2}.
We claim that for any $-2\sqrt{a}+\epsilon\leq c\leq 2\sqrt{a}-\epsilon$, any $\xi\in S^{N-1}$,
\begin{equation}
\label{step5-eq-1}
\inf_{|x|\leq 2L(\tilde\epsilon_0)}\tilde u(x,t;\xi, c, u_0,v_0)\ge \min\{\delta_{\tilde\epsilon_0},\tilde\delta_{\tilde\epsilon_0}\}\quad \forall\, t\ge T_{00}.
\end{equation}

To prove the claim, for any given $-2\sqrt{a}+\epsilon\leq c\leq 2\sqrt{a}-\epsilon$, $\xi\in S^{N-1}$,
let
$$
I=\{t> T_0{+2}\,\,\,|\, \sup_{|x|\leq 2L(\tilde\epsilon_0)}\tilde u(x,t;\xi, c, u_0,v_0)<\tilde\epsilon_0\}.
$$
Note that $I$ is an open set. By Lemma \ref{lem-2},
\begin{equation}
\label{step5-eq0-1}
\inf_{|x|\leq 2L(\tilde\epsilon_0)}\tilde u(x,t;\xi, c, u_0,v_0)\ge \delta_{\tilde\epsilon_0}\quad \forall\, t\not \in I\,\,\, {{\rm for}\, t>T_0+2}.
\end{equation}
Hence, if $I=\emptyset$, then
\begin{equation}
\label{step5-eq1-1}
\inf_{|x|\leq 2L(\tilde\epsilon_0)} \tilde u(x,t;\xi, c, u_0,v_0)\ge \delta_{\tilde\epsilon_0}\quad \forall\,\  t\ge T_0+2.
\end{equation}

If $I\not=\emptyset$, then $I=\cup (a_i,b_i)$. If $a_i\not =T_0+2$,  then
$$
\sup_{|x|\leq 2L(\tilde\epsilon_0)} \tilde u(x,a_i;\xi, c, u_0,v_0)=\tilde\epsilon_0
\quad{\rm and}\quad
\sup_{|x|\leq 2L(\tilde\epsilon_0)}\tilde u(x,t;\xi, c, u_0,v_0)<\tilde\epsilon_0\quad \forall\, t\in (a_i,b_i).
$$
 By Lemma \ref{lem-3},
 \begin{equation}
 \label{step5-eq2-1}
 \inf_{|x|\leq 2L(\tilde\epsilon_0)} \tilde u(x,t;\xi, c, u_0,v_0)\ge \tilde\delta_{\tilde\epsilon_0}\quad \forall\, t\in (a_i,b_i)
 \,\, {\rm for}\,\, a_i\not = T_0+2.
 \end{equation}
 If $a_i=T_0+2$,
 by the arguments in Lemma \ref{lem-3},  there holds
\begin{align*}
\tilde u(x,T_0+T(\tilde\epsilon_0)+3+k \tilde T_0;\xi, c, u_0,v_0)\geq  2^{k} \tilde\delta \bar \phi(x;\xi, c,\bar a) \quad \forall\,\ x\in D_{l}
\end{align*}
for $k=0,1,2,\cdots, k_0$. This implies that $b_i\le T_{00}$. This together with \eqref{step5-eq0-1},
\eqref{step5-eq1-1}, and \eqref{step5-eq2-1}
implies \eqref{step5-eq-1}.

By \eqref{step5-eq-1} and $\tilde u(x,t;\xi, c, u_0,v_0)=u(x+ct\xi, t;u_0,v_0)$, we have for any $-2\sqrt{a}+\epsilon\leq c\leq 2\sqrt{a}-\epsilon$, any $\xi\in S^{N-1}$,
\begin{equation*}
\label{new-step5-eq1-1}
\inf_{|x-ct\xi|\leq 2L(\tilde\epsilon_0)} u(x,t; u_0,v_0)\ge \min\{\delta_{\tilde\epsilon_0},\tilde\delta_{\tilde\epsilon_0}\}\quad \forall\, t\ge T_{00}.
\end{equation*}
Thus for any $t\ge T_{00}$, any $|x|\leq (2\sqrt{a}-\epsilon)t$, there exist $c=\frac{|x|}{t}$ and $\xi=\frac{x}{|x|}$ such that $|x-ct\xi|\leq 2L(\tilde\epsilon_{0})$, it then holds that
\begin{equation*}
\label{new-step5-eq2-1}
u(x,t; u_0,v_0)\ge \min\{\delta_{\tilde\epsilon_0},\tilde\delta_{\tilde\epsilon_0}\},
\end{equation*}
which implies that
$$
\inf_{|x|\leq (2\sqrt{a}-\epsilon)t} u(x,t; u_0,v_0)\ge \min\{\delta_{\tilde\epsilon_0},\tilde\delta_{\tilde\epsilon_0}\} \quad \forall\,\, t\ge T_{00}.
$$
Hence,
$$
\liminf_{t\to\infty}\inf_{|x|\leq (2\sqrt{a}-\epsilon)t} u(x,t; u_0,v_0)\ge \min\{\delta_{\tilde\epsilon_0},\tilde\delta_{\tilde\epsilon_0}\}.
$$
\eqref{lower-cp-u-1-1} is thus proved.

Finally, we prove that for any $0<\epsilon<\sqrt{a}$,
\begin{equation}\label{lower-cp-v-1-1}
\liminf_{t\to\infty}\inf_{|x|\leq (2\sqrt{a}-\epsilon)t}v(x,t;u_0,v_0)>0.
\end{equation}

Suppose by contraction that the result does not hold. Then there are constant $0<\epsilon<\sqrt{a}$ and a sequence $\{(x_{n}, t_{n})\}_{n\in \N}$ such that $t_{n}\to\infty$ as $n\to\infty$, $|x_{n}|\leq (2\sqrt{a}-\epsilon)t_{n}$, and
\begin{equation}\label{v-seq-0}
v(x_{n}, t_{n}; u_0, v_0)\to 0 \quad {\rm as }\,\,\,\, n\to\infty.
\end{equation}
For every $n\geq 1$, let us define
$$
u_{n}(x,t)=u(x+x_{n}, t+t_{n};u_0, v_0),\quad {\rm and} \quad v_{n}(x,t)=v(x+x_{n}, t+t_{n}; u_0, v_0)
$$
for every $x\in\R^{N}$, $t\geq -t_{n}$. By a prior estimates for parabolic equations, without loss of generality, we may assume that $(u_{n}(x,t),v_{n}(x,t))\to (u^{*}(x,t), v^{*}(x, t))$ locally uniformly in $C^{2,1}(\R^{N}\times\R)$. Furthermore, $(u^{*}(t,x), v^{*}(t,x))$ is an entire solution of
\begin{equation*}\label{u-Eq-1}
\begin{cases}
u_{t}=\Delta u - \chi\nabla\cdot(u\nabla v)+ u(a-bu),\quad  x\in\R^N, \,\, t\in \R\\
{v_t}=\Delta v-\lambda v+\mu u,\quad x\in\R^N, \,\, t\in\R.
\end{cases}
\end{equation*}

Choose $0<\tilde \epsilon<\epsilon$. For every $x\in \R^{N}$ and $t\in\R$, we have
\begin{align*}
|x+x_{n}|&\leq |x|+|x_{n}|\leq |x|+(2\sqrt{a}-\epsilon)t_{n}\cr
&=(2\sqrt{a}-\tilde\epsilon)(t_{n}+t)-(\epsilon-\tilde \epsilon)(t_{n}-\frac{|x|-(2\sqrt{a}-\tilde \epsilon)t}{\epsilon-\tilde \epsilon})\cr
&\leq (2\sqrt{a}-\tilde\epsilon)(t_{n}+t)
\end{align*}
whenever $t_{n}\geq \frac{|x|+(2\sqrt{a}-\tilde \epsilon)t}{\epsilon-\tilde \epsilon}$.  By \eqref{lower-cp-u-1-1},
$$
u^{*}(x, t)=\lim_{n\to\infty} u(x+x_{n}, t+t_{n}; u_0, v_0)\geq \liminf_{s\to\infty}\inf_{|y|\leq (2\sqrt{a}-\tilde \epsilon)s}u(y,s;u_0,v_0)>0
$$
for every $(x,t)\in \R^{N}\times \R$. It follows from comparison principle for parabolic equations that $v^{*}(x,t)>0$ for every $(x,t)\in \R^{N}\times \R$.
In particular, $v^{*}(0,0)>0$, which contradicts to \eqref{v-seq-0}.
\end{proof}

\subsection{Proof of Theorem \ref{spreading-thm-1} (1)}

In this subsection, we prove Theorem \ref{spreading-thm-1} (1). Throughout this subsection, let $\xi\in S^{N-1}$ and $(u_{0}, v_{0})\in C_{fl}^{+}(\xi)\times C_{fl}^{+,1}(\xi)$ be fixed.

\begin{proof}[Proof of Theorem \ref{spreading-thm-1} (1)]

We first prove that
for any $0<\epsilon<\sqrt{a}$,
\begin{equation}\label{lower-fl-u-2}
\liminf_{t\to\infty}\inf_{x\cdot\xi\leq (2\sqrt{a}-\epsilon)t}u(x,t;u_0,v_0)>0.
\end{equation}
Let $\tilde u(x,t)=u(x+(2\sqrt{a}-\epsilon)t\xi, t)$ and $\tilde v(x,t)=v(x+(2\sqrt{a}-\epsilon)t\xi, t)$. Then $(\tilde u(x,t), \tilde v(x,t))$ solves \eqref{tilde-u-v-eq} with $c$ being replaced by $2\sqrt{a}-\epsilon$. $(\tilde u(x,t;\xi,  u_0,v_0), \tilde v(x,t;\xi,  u_0,v_0))$ denotes the classical solution of \eqref{tilde-u-v-eq} with $c$ being replaced by $2\sqrt{a}-\epsilon$ and $(\tilde u(x,0;\xi,  u_0,v_0), \tilde v(x,0;\xi,  u_0,v_0))=(u_0, v_0)$.
Let $T_{0}=T_{0}(u_0,v_0)$ and $\tilde \epsilon_0$ be as in Lemma \ref{lem-2.2} and Lemma \ref{lem-3}, respectively. Let $T(\tilde \epsilon_0)$ be such that \eqref{T-choice} holds.
Let
$$
\tilde \delta=\inf_{x\cdot\xi\leq 2L(\tilde\epsilon_0)} \tilde u(x,T_0+T(\tilde \epsilon_0)+3;\xi, u_0,v_0).
$$
Since $\liminf_{x\cdot\xi\to -\infty}u_{0}(x)>0$, {$\tilde \delta>0$}. Let
$$
k_0=\inf\{k\in \Z^+\,|\, 2^{k}{\tilde \delta}\ge \tilde\epsilon_0\}
\quad {\rm and}\quad
T_{00}=T_0+T(\tilde\epsilon_0){+3}+k_0  \tilde T_0.
$$
where $ \tilde T_0\geq 1$ is as in Lemma \ref{lem-2}.
 By the similar arguments used in the proof of \eqref{step5-eq-1}, we can prove that  for any ball $B_{2L(\tilde \epsilon_0)}\subset \{x\,\,|\,\,x\cdot\xi<2L(\tilde \epsilon_0)\}$ with radius $2L(\tilde \epsilon_0)$, it holds that
\begin{equation*}
\label{step5-eq-2}
\inf_{x\in B_{2L(\tilde \epsilon_0)}}\tilde u(x,t;\xi, u_0,v_0)\ge \min\{\delta_{\tilde\epsilon_0},\tilde\delta_{\tilde\epsilon_0}\}\quad \forall\, t\ge T_{00}.
\end{equation*}

For any $x\in \{x\,\,|\,\,x\cdot\xi<2L(\tilde \epsilon_0)\}$, there exists a ball $B_{2L(\tilde \epsilon_0)}\subset \{x\,\,|\,\,x\cdot\xi<2L(\tilde \epsilon_0)\}$ such that $x\in B_{2L(\tilde \epsilon_0)}$, we then obtain that
$$
\tilde u(x,t;\xi, u_0,v_0)\geq \inf_{x\in B_{2L(\tilde \epsilon_0)}}\tilde u(x,t;\xi, u_0,v_0)\ge \min\{\delta_{\tilde\epsilon_0},\tilde\delta_{\tilde\epsilon_0}\}\quad \forall\, t\ge T_{00},
$$
which implies that
\begin{equation}\label{inf-tilde-u-2}
\inf_{x\cdot\xi<2L(\tilde \epsilon_0)} \tilde u(x,t;\xi, u_0,v_0)\geq \min\{\delta_{\tilde\epsilon_0},\tilde\delta_{\tilde\epsilon_0}\}>0 \quad \forall\, t\ge T_{00}.
\end{equation}

By \eqref{inf-tilde-u-2} and $\tilde u(x,t;\xi, u_0,v_0)=u(x+(2\sqrt{a}-\epsilon)t\xi, t; u_0,v_0)$, we have
\begin{equation*}
\inf_{x\cdot\xi< (2\sqrt{a}-\epsilon)t+2L(\tilde\epsilon_0)} u(x,t; u_0,v_0)\ge \min\{\delta_{\tilde\epsilon_0},\tilde\delta_{\tilde\epsilon_0}\}\quad \forall\, t\ge T_{00}.
\end{equation*}
Hence,
$$
\liminf_{t\to\infty} \inf_{x\cdot\xi\leq (2\sqrt{a}-\epsilon)t} u(x,t; u_0,v_0)\ge \min\{\delta_{\tilde\epsilon_0}, \tilde\delta_{\tilde\epsilon_0}\}.
$$
\eqref{lower-fl-u-2} is thus proved.

Finally, it can be proved by the similar arguments used in proving \eqref{lower-cp-v-1-1} that for any $0<\epsilon<\sqrt{a}$,
\begin{equation*}\label{lower-cp-v-2-2}
\liminf_{t\to\infty}\inf_{x\cdot\xi\leq (2\sqrt{a}-\epsilon)t}v(x,t;u_0,v_0)>0.
\end{equation*}
\end{proof}

\subsection{Proof of Theorem \ref{spreading-thm-2} (1)}

In this subsection, we prove Theorem \ref{spreading-thm-2} (1). Throughout this subsection, let $\xi\in S^{N-1}$ and $(u_{0}, v_{0})\in C^{+}(\xi)\times C^{+,1}(\xi)$ be fixed.
\begin{proof}[Proof of Theorem \ref{spreading-thm-2} (1)]

We first prove that
for any $0<\epsilon<\sqrt{a}$,
\begin{equation}\label{lower-cp-u-3}
\liminf_{t\to\infty}\inf_{|x\cdot\xi|\leq (2\sqrt{a}-\epsilon)t}u(x,t;u_0,v_0)>0.
\end{equation}

Let $T_{0}=T_{0}(u_0,v_0)$ and $\tilde \epsilon_0$ be as in Lemma \ref{lem-2.2} and Lemma \ref{lem-3}, respectively. Let $T(\tilde \epsilon_0)$ be such that \eqref{T-choice} holds.
For any $-2\sqrt{a}+\epsilon\leq c\leq 2\sqrt{a}-\epsilon$,
let
$$
\tilde \delta:=\tilde \delta(\xi,c)=\inf_{|x\cdot\xi|\leq 2L(\tilde\epsilon_0)} \tilde u(x,T_0+T(\tilde \epsilon_0)+3;\xi, c, u_0, v_0).
$$
Since there exists $r>0$ such that $\inf_{|x\cdot\xi|<r}u_{0}(x)>0$, $\tilde \delta>0$. Let
$$
k_0=\inf\{k\in \Z^+\,|\, 2^{k}{\tilde \delta}\ge \tilde\epsilon_0\}
\quad {\rm and}\quad
T_{00}=T_0+T(\tilde\epsilon_0){+3}+k_0  \tilde T_0.
$$
where $ \tilde T_0\geq 1$ is as in lemma \ref{lem-2}.
By the similar arguments used in the proof of \eqref{step5-eq-1}, we can prove that 
for any $-2\sqrt{a}+\epsilon\leq c\leq 2\sqrt{a}-\epsilon$, any ball $B_{2L(\tilde \epsilon_0)}\subset \{x\,\,|\,\,|x\cdot\xi|<2L(\tilde \epsilon_0)\}$ with radius $2L(\tilde \epsilon_0)$, it holds that
\begin{equation*}
\label{step5-eq-3}
\inf_{x\in B_{2L(\tilde \epsilon_0)}}\tilde u(x,t;\xi, c, u_0,v_0)\ge \min\{\delta_{\tilde\epsilon_0},\tilde\delta_{\tilde\epsilon_0}\}\quad \forall\, t\ge T_{00}.
\end{equation*}


For any $x\in \{x\,\,|\,\,|x\cdot\xi|<2L(\tilde \epsilon_0)\}$, there exists a ball $B_{2L(\tilde \epsilon_0)}\subset \{x\,\,|\,\,|x\cdot\xi|<2L(\tilde \epsilon_0)\}$ such that $x\in B_{2L(\tilde \epsilon_0)}$, we then obtain that for any $-2\sqrt{a}+\epsilon\leq c\leq 2\sqrt{a}-\epsilon$,
$$
\tilde u(x,t;\xi, c, u_0,v_0)\geq \inf_{x\in B_{2L(\tilde \epsilon_0)}}\tilde u(x,t;\xi, c, u_0,v_0)\ge \min\{\delta_{\tilde\epsilon_0},\tilde\delta_{\tilde\epsilon_0}\}\quad \forall\, t\ge T_{00},
$$
which implies that for any $-2\sqrt{a}+\epsilon\leq c\leq 2\sqrt{a}-\epsilon$,
\begin{equation}\label{inf-tilde-u-1-3}
\inf_{|x\cdot\xi|<2L(\tilde \epsilon_0)} \tilde u(x,t;\xi, c, u_0,v_0)\geq \min\{\delta_{\tilde\epsilon_0},\tilde\delta_{\tilde\epsilon_0}\}>0 \quad \forall\, t\ge T_{00}.
\end{equation}

By \eqref{inf-tilde-u-1-3} and $\tilde u(x,t;\xi, c, u_0,v_0)=u(x+ct\xi, t; u_0,v_0)$, we have for any $-2\sqrt{a}+\epsilon\leq c\leq 2\sqrt{a}-\epsilon$,
\begin{equation*}
\label{new-step-eq2-3}
\inf_{|x\cdot\xi-ct|< 2L(\tilde\epsilon_0)} u(x,t; u_0,v_0)\ge \min\{\delta_{\tilde\epsilon_0},\tilde\delta_{\tilde\epsilon_0}\}\quad \forall\, t\ge T_{00}.
\end{equation*}
For any $t\ge T_{00}$, any $x\in \{x\,| \, |x\cdot\xi|\leq (2\sqrt{a}-\epsilon)t\}$, there exists $c=\frac{x\cdot\xi}{t}$ such that $x\in
\{x\,| \, |x\cdot\xi-ct|<2L(\tilde \epsilon_{0})\}$, it then holds that
\begin{equation*}
u(x, t; u_0,v_0)\ge \min\{\delta_{\tilde\epsilon_0},\tilde\delta_{\tilde\epsilon_0}\},
\end{equation*}
which implies that
$$
\inf_{|x\cdot\xi|\leq (2\sqrt{a}-\epsilon)t} u(x,t; u_0,v_0)\ge \min\{\delta_{\tilde\epsilon_0},\tilde\delta_{\tilde\epsilon_0}\} \quad \forall\,\, t\ge T_{00}.
$$
Hence,
$$
\liminf_{t\to\infty}\inf_{|x\cdot\xi|\leq (2\sqrt{a}-\epsilon)t} u(x, t; u_0,v_0)\ge \min\{\delta_{\tilde\epsilon_0},\tilde\delta_{\tilde\epsilon_0}\}.
$$
\eqref{lower-cp-u-3} is thus proved.

Finally, 
it can be proved by the similar arguments used in proving \eqref{lower-cp-v-1-1} that for any $0<\epsilon<\sqrt{a}$,
\begin{equation*}\label{lower-cp-v-2}
\liminf_{t\to\infty}\inf_{|x\cdot\xi|\leq (2\sqrt{a}-\epsilon)t}v(x,t;u_0,v_0)>0.
\end{equation*}
\end{proof}

\section{Upper bounds of spreading speeds}

This section is devoted to the study of upper bounds of the spreading speeds of  global classical solutions of \eqref{Main-Eq} with different initial functions and prove Theorem \ref{spreading-thm} (2), Theorem \ref{spreading-thm-1} (2) and Theorem \ref{spreading-thm-2} (2).
Throughout this section, we assume that $b>\frac{N\mu\chi}{4}$.

First, we present a lemma.

\begin{lem}
\label{upper-bound-lm1}
Let $w=u+\frac{\chi}{2\mu}|\nabla v|^2$.
Then
$$
w_t\le \Delta w+a w.
$$
\end{lem}

\begin{proof}
Similar arguments to those used in the proof of \cite[Theorem 1.2]{ShXu} yield that
\begin{align*}
\label{new-eq0}
\frac{d}{dt}\big[u+ \frac{\chi}{2\mu}\left|\nabla v \right|^2 \big]
&\leq \Delta \big[u+\frac{\chi}{2\mu}\left|\nabla v \right|^2\big]-\frac{\chi\lambda}{\mu}\left|\nabla v\right|^2-\Big(b-\frac{N\mu\chi}{4}\Big)u^2+au.
\end{align*}
Since $b>\frac{N\mu\chi}{4}$, then
\begin{equation}\label{u-gv}
\frac{d}{dt}\big[u+ \frac{\chi}{2\mu}\left|\nabla v \right|^2 \big]\leq \Delta \big[u+\frac{\chi}{2\mu}\left|\nabla v \right|^2\big]+a\big[u+\frac{\chi}{2\mu}\left|\nabla v \right|^2\big]
\end{equation}
The lemma then follows from \eqref{u-gv}.
\end{proof}

We now prove Theorem \ref{spreading-thm} (2)

\begin{proof}[Proof of Theorem \ref{spreading-thm} (2)] First of all, for any given $(u_{0}, v_{0})\in C_{cp}^{+}\times C_{cp}^{+,1}$ and
$0<k<\sqrt a$, let $M>0$ be such that
\begin{equation*}
\label{new-new-eq1}
u_0(x)+\frac{\chi}{2\mu}|\nabla v_0(x)|^2\le \min\{M e^{-kx\cdot\xi},\,\  \xi\in S^{N-1}\}\quad \forall\, x\in\R^N.
\end{equation*}
Let
$$
c=\frac{k^2+a}{k},
$$
and
$$
U(x,t,\xi)=M e^{-k(x\cdot\xi-ct)}.
$$
Write $u=u(x,t;u_0,v_0)$, $v=v(x,t;u_0,v_0)$. Let $w=u+\frac{\chi}{2\mu}|\nabla v|^2$. By Lemma \ref{upper-bound-lm1},
$$
w_t\le \Delta w+aw.
$$
It follows from comparison principle for parabolic equations that
\begin{equation}\label{sup-solu-1-1}
u(x,t;u_0,v_0)\leq U(x,t,\xi)\quad \forall\, x\in\R^N,\,\ t>0,\,\ \xi\in S^{N-1}.
\end{equation}
Let $\xi=\frac{x}{|x|}$, then
$$
u(x,t;u_0,v_0)\leq M e^{-k(|x|-ct)}\quad \forall\, x\in\R^N,\,\ t>0.
$$
For any $\varepsilon>0$, there exists $0<k<\sqrt{a}$ such that $2\sqrt{a}+\varepsilon>c$, it then holds that
\begin{equation}\label{u-upp}
\lim_{t\to\infty}\sup_{|x|\geq (2\sqrt{a}+\varepsilon)t}u(x,t;u_0,v_0)=0.
\end{equation}

Next, we prove that for any $\varepsilon>0$,
\begin{equation}\label{v-upp}
\lim_{t\to\infty}\sup_{|x|\geq (2\sqrt{a}+\varepsilon)t}v(x,t;u_0,v_0)=0.
\end{equation}
Let $d\geq \frac{\mu M}{a+\lambda}$ be such that
$$
v_0(x)\le \min\{d e^{-kx\cdot\xi},\,\  \xi\in S^{N-1}\} \quad \forall\, x\in\R^N.
$$
By the second equation of \eqref{Main-Eq} and \eqref{sup-solu-1-1},
$$
{v_t}=\Delta v -\lambda v+\mu u\leq \Delta v -\lambda v+\mu Me^{-k(x\cdot\xi-ct)}.
$$
Direct computation yields that $de^{-k(x\cdot\xi-ct)}$ satisfies that
$$
\frac{\partial}{\partial t}(de^{-k(x\cdot\xi-ct)})\geq \Delta (de^{-k(x\cdot\xi-ct)})-\lambda (de^{-k(x\cdot\xi-ct)})+\mu Me^{-k(x\cdot\xi-ct)}.
$$
It follows from comparison principle for parabolic equations again that
\begin{equation*}\label{sup-solu-1}
v(x,t;u_0,v_0)\leq de^{-k(x\cdot\xi-ct)}\quad \forall\, x\in\R^N,\,\ t>0,\,\ \xi\in S^{N-1}.
\end{equation*}
Similar arguments as in deriving \eqref{u-upp} yield that \eqref{v-upp} holds.
\end{proof}

Next, we prove Theorem \ref{spreading-thm-1} (2).

\begin{proof}[Proof of Theorem \ref{spreading-thm-1} (2)]
For any given $\xi\in S^{N-1}$, $(u_{0}, v_{0})\in C_{fl}^{+}(\xi)\times C_{fl}^{+,1}(\xi)$ and
$0<k<\sqrt a$, let $$c=\frac{k^2+a}{k}, $$
and $M>0$ be such that
\begin{equation*}
\label{new-new-eq1}
u_0(x)+\frac{\chi}{2\mu}|\nabla v_0(x)|^2\le M e^{-kx\cdot\xi} \quad \forall\,\ x\in\R^N.
\end{equation*}
Let $d\geq \frac{\mu M}{a+\lambda}$ be such that
$$
v_0(x)\le d e^{-kx\cdot\xi} \quad \forall\,\ x\in\R^N.
$$
By similar arguments as those in Theorem \ref{spreading-thm} (2), we can prove that
$$
u(x,t;u_0,v_0)\leq M e^{-k(x\cdot\xi-ct)}\quad \forall\, x\in\R^N,\,\ t>0
$$
and
$$
v(x,t;u_0,v_0)\leq d e^{-k(x\cdot\xi-ct)}\quad \forall\, x\in\R^N,\,\ t>0
$$
For any $\varepsilon>0$, there exists $0<k<\sqrt{a}$ such that $2\sqrt{a}+\varepsilon>c$, Theorem \ref{spreading-thm-1} (2) thus follows.
\end{proof}

At the end of this section, we prove Theorem \ref{spreading-thm-2} (2).

\begin{proof}[Proof of Theorem \ref{spreading-thm-2} (2)]

For any given $\xi\in S^{N-1}$, $(u_{0}, v_{0})\in C^{+}(\xi)\times C^{+,1}(\xi)$ and
$0<k<\sqrt a$, let
$$c=\frac{k^2+a}{k},$$
and $M>0$ be such that
\begin{equation*}
\label{new-new-eq1}
u_0(x)+\frac{\chi}{2\mu}|\nabla v_0(x)|^2\le \min\{M e^{-kx\cdot\xi},\, M e^{kx\cdot\xi}\} \quad \forall\,\ x\in\R^N.
\end{equation*}
Let $d\geq \frac{\mu M}{a+\lambda}$ be such that
$$
v_0(x)\le \min\{d e^{-kx\cdot\xi},\, d e^{kx\cdot\xi}\} \quad \forall\,\ x\in\R^N.
$$
By the similar arguments as those in Theorem \ref{spreading-thm} (2), we can prove that
$$
u(x,t;u_0,v_0)\leq M e^{-k(x\cdot\xi-ct)}\quad \forall\, x\in\R^N,\,\ t>0,
$$

$$
u(x,t;u_0,v_0)\leq M e^{k(x\cdot\xi+ct)}\quad \forall\, x\in\R^N,\,\ t>0,
$$

$$
v(x,t;u_0,v_0)\leq d e^{-k(x\cdot\xi-ct)}\quad \forall\, x\in\R^N,\,\ t>0,
$$
and
$$
v(x,t;u_0,v_0)\leq d e^{k(x\cdot\xi+ct)}\quad \forall\, x\in\R^N,\,\ t>0.
$$
It then follows that
$$
u(x,t;u_0,v_0)\leq M e^{-k(|x\cdot\xi|-ct)}\quad \forall\, x\in\R^N,\,\ t>0,
$$
and
$$
v(x,t;u_0,v_0)\leq d e^{-k(|x\cdot\xi|-ct)}\quad \forall\, x\in\R^N,\,\ t>0.
$$
For any $\varepsilon>0$, there exists $0<k<\sqrt{a}$ such that $2\sqrt{a}+\varepsilon>c$, Theorem \ref{spreading-thm-2} (2) thus follows.
\end{proof}


\end{document}